\documentclass[a4paper,10pt]{amsart}
\usepackage[utf8x]{inputenc}
\usepackage{amsthm}
\title[The inversion number and the major index]{The inversion number and the major index\\ are asymptotically jointly normally distributed\\ on words}
\author{Marko Thiel}
\newtheorem{lemma}{Lemma}
\newtheorem{theorem}{Theorem}
\newtheorem{claim}{Claim}

\newtheorem{corollary}{Corollary}
\allowdisplaybreaks
\begin{document}

\begin{abstract}
In a recent paper \cite{baxter}, Baxter and Zeilberger show that the two most important Mahonian statistics, the inversion number and the major index, are asymptotically independently normally distributed on permutations. In another recent paper \cite{canfield}, Canfield, Janson and Zeilberger prove the result, already known to statisticians, that the Mahonian distribution is asymptotically normal on words. This leaves one question unanswered: What, asymptotically, is the joint distribution of the inversion number and the major index on words? We answer this question by establishing convergence to a bivariate normal distribution. 
\end{abstract}
\maketitle
\section{Introduction}
Let $S_n$ be the set of permutations of $n$ objects, that is the set of bijections from $[n]=\{1,2,\ldots,n\}$ to itself. Define the inversion number as $inv(\pi)=|\{(i,j)\in[n]\times[n]\mid i<j \text{ and } \pi(i)>\pi(j)\}|$ and the major index as $\sum_{i\in Des(\pi)}i$, where $Des(\pi)=\{i\in[n-1]\mid\pi(i)>\pi(i+1)\}$ is the descent set of $\pi$. It is a classical result due to MacMahon \cite{macmahon} that these have the same distribution, termed the Mahonian distribution in his honour. Their probabilty generating function with respect to the uniform probability on $S_n$ is
\begin{equation*}
\sum_{\pi\in S_n}q^{inv(\pi)}\Pr(\pi)=\sum_{\pi\in S_n}q^{maj(\pi)}\Pr(\pi)=\frac{1}{n!}\prod_{i=1}^n\frac{1-q^i}{1-q}\text{.}
\end{equation*}
There is a bijection $\phi$ from $S_n$ to itself due to Foata \cite{foata} that satisfies $inv(\phi(\pi))=maj(\pi)$, again proving that $inv$ and $maj$ are equidistributed on $S_n$. This bijection fixes the last letter of a permutation, so $inv$ and $maj$ are even equidistributed on $S_{n,i}=\{\pi\in S_n\mid\pi(n)=i\}$ for all $i\in[n]$.
\\
\\
In a recent paper \cite{baxter}, Baxter and Zeilberger show that $inv$ and $maj$ are asymptotically jointly independently normally distributed on $S_n$, that is
\begin{multline*}
\Pr\left(\frac{inv-\mu}{\sigma}\leq x \text{ and } \frac{maj-\mu}{\sigma}\leq y\right)\longrightarrow \frac{1}{2\pi}\int_{-\infty}^x e^{-t^2/2}\mathrm{d}t\int_{-\infty}^y e^{-t^2/2}\mathrm{d}t\quad \\ \text{ as $n$}\rightarrow\infty\text{,}
\end{multline*}
where $\mu=\mu_n=\mathbb{E}(inv)=\mathbb{E}(maj)$ is the mean and $\sigma^2=\sigma_n^2=\mathbb{E}((inv-\mu)^2)=\mathbb{E}((maj-\mu)^2)$ is the variance of the Mahonian distribution.
\\
\\
A simple generalization of permutations are words. A word of length $n$ is just a function $w:[n]\rightarrow[d]$. It can be regarded as a permutation of a multiset, sometimes called a \textit{permatution}. Now the inversion number and the major index can be defined for words in the same way as for permutations. Foata's bijection $\phi$ also extends to words \cite{foata}, so if $A$ is a multiset, and $S_A$ the set of permatutions of it, $inv$ and $maj$ are equidistributed on $S_A$. In fact, as for permutations, they are even equidistributed on $S_{A,i}=\{w\in S_A\mid w\big(\sum_{j=1}^da_j\big)=i\}$.
If $A=\{1^{a_1},2^{a_2},\dots,d^{a_d}\}$ is the multiset containing $a_1$ occurences of 1, $a_2$ occurences of 2, \dots, $a_d$ occurences of $d$, their probability generating function on $S_A$ is
\begin{equation*}
\sum_{w\in S_A}q^{inv(w)}\Pr(w)=\sum_{w\in S_A}q^{maj(w)}\Pr(w)=\frac{\binom{a_1+\cdots+a_d}{a_1,\ldots,a_d}_q}{\binom{a_1+\cdots+a_d}{a_1,\ldots,a_d}}\text{,}
\end{equation*}
where $\binom{a_1+\cdots+a_d}{a_1,\ldots,a_d}_q=\frac{[a_1+\cdots+a_d]_q!}{[a_1]_q!\cdots[a_d]_q!}$ is the $q$-multinomial coefficient, defined in terms of the $q$-factorials $[n]_q!=[1]_q\cdots[n]_q$ and the $q$-numbers $[n]_q=(1-q^n)/(1-q)$.
\\
\\
In a recent paper \cite{canfield}, Canfield, Janson and Zeilberger use methods from experimental mathematics as well as Fourier analytic techniques to give multiple proofs of the fact, already known to statisticians, that the Mahonian distribution is asymptotically normal on $S_A$ as well. This leaves one question unanswered: What, asymptotically, is the joint probability distribution of $inv$ and $maj$ on $S_A$?
\section{The Main Result}
\begin{theorem}
The inversion number and the major index are asymptotically jointly normally distributed on words. That is, if $A$ is the multiset containing $m_1a$ occurences of 1, $m_2a$ occurences of 2, \dots, $m_da$ occurences of $d$, then $(\frac{inv-\mu}{\sigma},\frac{maj-\mu}{\sigma})$ on $S_A$ tends to a bivariate normal distribution as $a\rightarrow\infty$. The correlation coefficient of that distribution is $$\frac{\sum_{1\leq i<j\leq d}m_im_j^2-\sum_{1\leq i<j\leq d}m_i^2m_j}{\sum_{1\leq i<j\leq d}(m_im_j^2+m_i^2m_j)+2\sum_{1\leq i<j<k\leq d}m_im_jm_k}\text{.}$$
\end{theorem}
\section{Proof}
To prove this, we use the method of moments. Recall that two random variables $X$ and $Y$ jointly converge to a bivariate normal distribution if and only if their mixed moments $\mathbb{E}(X^rY^s)$ converge to the mixed moments of the bivariate normal distribution \cite{feller}. We shall instead consider factorial moments $\mathbb{E}(X^{\underline{r}}Y^{\underline{r}})$, with the factorial powers $x^{\underline{k}}=x(x-1)\cdots(x-k+1)$, derive a recurrence for them, and check that up to the leading terms it agrees with a well-known recurrence for the mixed moments of a bivariate normal distribution, thereby proving the result.
\\
\\
Consider the effect that removing the last element from a word has on its inversion number and its major index. The inversion number will decrease by the number of letters smaller than the last letter in the word, and the major index will decrease only if the second to last letter was bigger than the last, by the position of the second to last letter. In the language of generating functions this is
\begin{equation}
F(\textbf{a},i)(p,q)=p^{\sum_{j=i+1}^da_i}\bigg(\sum_{j=1}^iF(\textbf{a}-\textbf{e}_i,j)(p,q)+q^{\sum_{j=1}^da_j-1}\sum_{j=i+1}^dF(\textbf{a}-\textbf{e}_i,j)(p,q)\bigg)
\end{equation}
where $F(\textbf{a},i)(p,q)=F((a_1,a_2,\dots,a_d),i)(p,q)=\sum_{w\in S_{A,i}}p^{inv(w)}q^{maj(w)}$ is the double generating function of $inv$ and $maj$ on the set of permatutions of the multiset $A=\{1^{a_1},2^{a_2},\dots,d^{a_d}\}$ that end with the letter $i$, and $\textbf{e}_i$ is the $i$-th unit vector, having a 1 in its $i$-th coordinate and zeroes in all others.
The mean of $inv$ on $S_A$ is easily seen to be $e_2(\textbf{a})/2$, where $e_2(\textbf{a})=\sum_{1\leq i<j\leq d}a_ia_j$ is the second elementary symmetric polynomial in $a_1,\dots,a_d$, and $maj$ has the same distribution, so the centralized probability generating function of $inv$ and $maj$ on $S_{A,i}$ is
\begin{equation*}
G(\textbf{a},i)(p,q)=\frac{F(\textbf{a},i)(p,q)}{\binom{\sum_{j=1}^da_j-1}{\textbf{a}-\textbf{e}_i}(pq)^{\frac{e_2(\textbf{a}-\textbf{e}_i)}{2}+\sum_{j=i+1}^da_j}}\text{,}
\end{equation*}
where $\binom{\sum_{j=1}^da_j}{\textbf{a}}=\binom{\sum_{j=1}^da_j}{a_1,\dots,a_d}$ is the multinomial coefficient. So the recurrence (1) translates to
\begin{multline*}
\binom{\sum_{j=1}^da_j}{\textbf{a}}(pq)^{\frac{e_2(\textbf{a})}{2}+\sum_{j=i+1}^da_j}G(\textbf{a}+\textbf{e}_i,i)(p,q)\\
=p^{\sum_{j=i+1}^da_j}\bigg(\sum_{j=1}^i\binom{\sum_{k=1}^da_k-1}{\textbf{a}-\textbf{e}_j}(pq)^{\frac{e_2(\textbf{a}-\textbf{e}_j)}{2}+\sum_{k=j+1}^da_k}G(\textbf{a},j)(p,q)\\
+q^{\sum_{j=1}^da_j-1}\sum_{j=i+1}^d\binom{\sum_{k=1}^da_k-1}{\textbf{a}-\textbf{e}_j}(pq)^{\frac{e_2(\textbf{a}-\textbf{e}_j)}{2}+\sum_{k=j+1}^da_k}G(\textbf{a},j)(p,q)\bigg)\text{,}
\end{multline*}
which simplifies to
\begin{multline}
G(\textbf{a}+\textbf{e}_i,i)(p,q)\\
=\sum_{j=1}^{i}\frac{a_j}{\sum_{k=1}^da_k}p^{\frac{1}{2}(\sum_{k=j+1}^da_k-\sum_{k=1}^{j-1}a_k)}q^{\frac{1}{2}(\sum_{k=j+1}^{i}a_k-\sum_{k=1}^{j-1}a_k-\sum_{k=i+1}^{d}a_k)}G(\textbf{a},j)(p,q)\\
+\sum_{j=i+1}^{d}\frac{a_j}{\sum_{k=1}^da_k}p^{\frac{1}{2}(\sum_{k=j+1}^da_k-\sum_{k=1}^{j-1}a_k)}q^{\frac{1}{2}(\sum_{k=j+1}^{d}a_k+\sum_{k=1}^{i}a_k-\sum_{k=i+1}^{j-1}a_k)}G(\textbf{a},j)(p,q)\text{.}
\end{multline}
Let $FM(\textbf{a},i,r,s)=\mathbb{E}\big((inv-\mu)^{\underline{r}}(maj-\mu)^{\underline{s}}\big)$ be the $(r,s)$-th mixed factorial moment of the random variables $inv-\mu$ and $maj-\mu$, where $\mu=\mathbb{E}(inv)=\mathbb{E}(maj)$, on $S_{A,i}$. Then
\begin{equation*}
\frac{\partial^r}{\partial p^r}\frac{\partial^s}{\partial q^s}G(\textbf{a},i)(p,q)\Big|_{p=q=1}=FM(\textbf{a},i,r,s)\text{.}
\end{equation*}
So we get the Taylor series
\begin{equation*}
G(\textbf{a},i)(1+p,1+q)=\sum_{r,s=0}^\infty\frac{FM(\textbf{a},i,r,s)}{r!s!}p^rq^s\text{.}
\end{equation*}
Thus (2) translates to
\begin{multline*}
\sum_{r,s=0}^\infty\frac{FM(\textbf{a}+\textbf{e}_i,i,r,s)}{r!s!}p^rq^s\\
\begin{aligned}
=\sum_{j=1}^i\frac{a_j}{\sum_{k=1}^da_k}(1+p)^{\frac{1}{2}(\sum_{k=j+1}^da_k-\sum_{k=1}^{j-1}a_k)}(1+q)^{\frac{1}{2}(\sum_{k=j+1}^{i}a_k-\sum_{k=1}^{j-1}a_k-\sum_{k=i+1}^{d}a_k)}\\
\cdot\sum_{r,s=0}^\infty\frac{FM(\textbf{a},j,r,s)}{r!s!}p^rq^s\\
+\sum_{j=i+1}^d\frac{a_j}{\sum_{k=1}^da_k}(1+p)^{\frac{1}{2}(\sum_{k=j+1}^da_k-\sum_{k=1}^{j-1}a_k)}(1+q)^{\frac{1}{2}(\sum_{k=j+1}^{d}a_k+\sum_{k=1}^{i}a_k-\sum_{k=i+1}^{j-1}a_k)}\\
\cdot\sum_{r,s=0}^\infty\frac{FM(\textbf{a},j,r,s)}{r!s!}p^rq^s\text{.}
\end{aligned}
\end{multline*}
Now expand the powers using the binomial theorem and compare coefficients to obtain
\begin{multline*}
FM(\textbf{a}+\textbf{e}_i,i,r,s)\\
\begin{aligned}
=\sum_{j=1}^i\frac{a_j}{\sum_{k=1}^da_k}&\sum_{r'=0}^r\sum_{s'=0}^s\binom{\frac{1}{2}(\sum_{k=j+1}^da_k-\sum_{k=1}^{j-1}a_k)}{r-r'}\\
\cdot& \binom{\frac{1}{2}(\sum_{k=j+1}^{i}a_k-\sum_{k=1}^{j-1}a_k-\sum_{k=i+1}^{d}a_k)}{s-s'}r^{\underline{r-r'}}s^{\underline{s-s'}}FM(\textbf{a},j,r',s')\\
+\sum_{j=i+1}^d&\frac{a_j}{\sum_{k=1}^da_k}\sum_{r'=0}^r\sum_{s'=0}^s\binom{\frac{1}{2}(\sum_{k=j+1}^da_k-\sum_{k=1}^{j-1}a_k)}{r-r'}\\
\cdot& \binom{\frac{1}{2}(\sum_{k=j+1}^{d}a_k+\sum_{k=1}^{i}a_k-\sum_{k=i+1}^{j-1}a_k)}{s-s'}r^{\underline{r-r'}}s^{\underline{s-s'}}FM(\textbf{a},j,r',s')\text{.}
\end{aligned}
\end{multline*}
So
\begin{multline}
\Big(\sum_{j=1}^da_j\Big)FM(\textbf{a}+\textbf{e}_i,i,r,s)-\sum_{j=1}^da_jFM(\textbf{a},j,r,s)\\
\begin{aligned}
=\sum_{j=1}^ia_j&\sum_{\{(r',s')\mid r'<r \text{ or }s'<s\}}\binom{\frac{1}{2}(\sum_{k=j+1}^da_k-\sum_{k=1}^{j-1}a_k)}{r-r'}\\
\cdot& \binom{\frac{1}{2}(\sum_{k=j+1}^{i}a_k-\sum_{k=1}^{j-1}a_k-\sum_{k=i+1}^{d}a_k)}{s-s'}r^{\underline{r-r'}}s^{\underline{s-s'}}FM(\textbf{a},j,r',s')\\
+&\sum_{j=i+1}^da_j\sum_{\{(r',s')\mid r'<r \text{ or }s'<s\}}\binom{\frac{1}{2}(\sum_{k=j+1}^da_k-\sum_{k=1}^{j-1}a_k)}{r-r'}\\
\cdot& \binom{\frac{1}{2}(\sum_{k=j+1}^{d}a_k+\sum_{k=1}^{i}a_k-\sum_{k=i+1}^{j-1}a_k)}{s-s'}r^{\underline{r-r'}}s^{\underline{s-s'}}FM(\textbf{a},j,r',s')\text{.}
\end{aligned}
\end{multline}
To really profit from this recurrence, we need the following lemma.
\begin{lemma}
$FM(\textbf{a},i,r,s)$ is a rational function in $a_1,\dots,a_d$, with denominator $\big(\sum_{j=1}^da_j-1\big)^{\underline{4(r+s)}}$.
\end{lemma}
\begin{proof}
We use a combinatorial argument similar to one given by Baxter and Zeilberger for permutations \cite{baxter}. For simplicity, we shall first consider ordinary mixed moments instead of factorial moments. Start with the equation
\begin{multline*}
\mathbb{E}(inv^rmaj^s)=\\
\frac{1}{\binom{\sum_{j=1}^da_j-1}{\textbf{a}-\textbf{e}_i}}\sum_{w\in S_{A,i}}\Bigg(\sum_{j,k=1}^{\sum_{j=1}^da_j}[j<k\text{ and }w(j)>w(k)]\Bigg)^r\\
\cdot\Bigg(\sum_{j,k=1}^{\sum_{j=1}^da_j}[w(j)>w(j+1)\text{ and }k\leq j]\Bigg)^s\text{,}
\end{multline*}
where $[P]=1$ if $P$ is true and $[P]=0$ otherwise. Now expand the powers and move the sum over $S_{A,i}$ and the division by $\binom{\sum_{j=1}^da_j-1}{\textbf{a}-\textbf{e}_i}$ to the right of all the resulting $2(r+s)$ summation symbols. Then there are only finitely many possibilities, independent of $a_1,\dots,a_d$, of choosing the relative order of the $2(r+s)$ indices of summation. Some of these may coincide, so say there are $l$ distinct indices $i_1,\dots,i_l$. Let $O=\{j\in[l]\mid i_j+1=i_{j+1}\}$, so $o=|O|$ is the number of distinct indices whose successor is also an index. Now there are only finitely many possibilities, independent of $a_1,\dots,a_d$, of choosing the relative order of the indices, then picking the set $O$, then specifying whether $i_l=\sum_{j=1}^da_j$, $i_l=\sum_{j=1}^da_j-1$ or $i_l<\sum_{j=1}^da_j-1$, and afterwards choosing the values taken by $w$ on $\{i_1,\dots,i_l\}\cup\{i_1+1,\dots,i_l+1\}$. 
For each of these scenarios, the summand takes the same value, either 0 or 1, on all of the words that belong to that scenario. The number of words belonging to a scenario is $\binom{\sum_{j=1}^da_j-l}{l-o-1}\binom{\sum_{j=1}^da_j-2l+o+1}{\textbf{a}-\textbf{v}}$ if $i_l=\sum_{j=1}^da_j$, $\binom{\sum_{j=1}^da_j-l-1}{l-o-1}\binom{\sum_{j=1}^da_j-2l+o}{\textbf{a}-\textbf{v}}$ if $i_l=\sum_{j=1}^da_j-1$, and $\binom{\sum_{j=1}^da_j-l-1}{l-o}\binom{\sum_{j=1}^da_j-2l+o-1}{\textbf{a}-\textbf{v}-\textbf{e}_i}$ if $i_l<\sum_{j=1}^da_j-1$, where $\textbf{v}$ is the vector of multiplicities of the values taken by $w$ on the positions of the indices and their successors as defined by the scenario.
Here the binomial coefficient counts the number of ways to choose the positions of the indices, while the multinomial coefficient counts the number of ways of choosing the values of $w$ at positions smaller than $\sum_{j=1}^da_j$ that are not an index or the successor of an index.
So the contribution to the sum of that scenario is either zero,
\begin{multline*}
\frac{\binom{\sum_{j=1}^da_j-l}{l-o-1}\binom{\sum_{j=1}^da_j-2l+o+1}{\textbf{a}-\textbf{v}}}{\binom{\sum_{j=1}^da_j-1}{\textbf{a}-\textbf{e}_i}}\text{, }\frac{\binom{\sum_{j=1}^da_j-l-1}{l-o-1}\binom{\sum_{j=1}^da_j-2l+o}{\textbf{a}-\textbf{v}}}{\binom{\sum_{j=1}^da_j-1}{\textbf{a}-\textbf{e}_i}}\text{,}\\
\text{ or }\frac{\binom{\sum_{j=1}^da_j-l-1}{l-o}\binom{\sum_{j=1}^da_j-2l+o-1}{\textbf{a}-\textbf{v}-\textbf{e}_i}}{\binom{\sum_{j=1}^da_j-1}{\textbf{a}-\textbf{e}_i}}\text{.} 
\end{multline*}
But all of these are rational functions with denominator $\big(\sum_{j=1}^da_j-1\big)^{\underline{2l-o}}$, so $\mathbb{E}(inv^rmaj^s)$ is a rational function with denominator $\big(\sum_{j=1}^da_j-1\big)^{\underline{4(r+s)}}$.
From this we can recover $\mathbb{E}((inv-\mu)^r(maj-\mu)^s)$ using the binomial theorem and $FM(\textbf{a},i,r,s)=\mathbb{E}((inv-\mu)^{\underline{r}}(maj-\mu)^{\underline{s}})$ using Stirling numbers \cite{graham}. So $FM(\textbf{a},i,r,s)$ is a rational function with denominator $\big(\sum_{j=1}^da_j-1\big)^{\underline{4(r+s)}}$.
\end{proof}
In fact, computational experiments suggest that the $FM(\textbf{a},i,r,s)$ are even \textit{polynomials} in $a_1,\dots,a_d$, but the author has been unable to prove this.\\
\\
So now we know that the $FM(\textbf{a},i,r,s)$ are of the form $ \frac{P(\textbf{a},i,r,s)}{(\sum_{j=1}^da_j-1)^{\underline{4(r+s)}}}$, where $P(\textbf{a},i,r,s)$ is a polynomial in the variables $a_1,\dots,a_d$. But what is the degree of $P(\textbf{a},i,r,s)$?\\
\\
\begin{lemma}
The degree of $P(\textbf{a},i,r,s)$ is at most $\frac{11}{2}(r+s)$.
\end{lemma}
\begin{proof}
Let $\textbf{a}=(a_1,\dots,a_d)$ be a vector of nonnegative integers and let $A=\{1^{a_1},\dots,d^{a_d}\}$. Define $X_{\textbf{a},i}=inv-\mu$ and $Y_{\textbf{a},i}=maj-\mu$ on $S_{A,i}$. Then
\begin{multline*}
|\mathbb{E}(X_{\textbf{a},i}^rY_{\textbf{a},i}^s)|\leq\mathbb{E}(|X_{\textbf{a},i}|^r|Y_{\textbf{a},i}|^s)\leq\mathbb{E}\big((|X_{\textbf{a},i}|^r)^{\frac{r+s}{r}}\big)^{\frac{r}{r+s}}\mathbb{E}\big((|Y_{\textbf{a},i}|^s)^\frac{r+s}{s}\big)^\frac{s}{r+s}\\
=\mathbb{E}(|X_{\textbf{a},i}|^{r+s})=\mathbb{E}(X_{\textbf{a},i}^{r+s})
\end{multline*}
if $r+s$ is even, by Hölder's inequality. From the the paper by Canfield, Janson and Zeilberger \cite{canfield}, we know that $\mathbb{E}(X_{\textbf{a},i}^{r+s})$ is a polynomial of degree $\frac{3}{2}(r+s)$, so that if $\mathbb{E}(X_{\textbf{a},i}^rY_{\textbf{a},i}^s)$ is written as a rational function with denominator $\big(\sum_{j=1}^da_j-1\big)^{\underline{4(r+s)}}$, the degree of the numerator is at most $4(r+s)+\frac{3}{2}(r+s)=\frac{11}{2}(r+s)$. $FM(\textbf{a},i,r,s)$ can be recovered from the $\mathbb{E}(X_{\textbf{a},i}^rY_{\textbf{a},i}^s)$ using Stirling numbers \cite{graham}, so the degree of $P(\textbf{a},i,r,s)$ is also at most $\frac{11}{2}(r+s)$.
If $r+s$ is odd, 
\begin{displaymath}
\left|\mathbb{E}(X_{\textbf{a},i}^rY_{\textbf{a},i}^s)\right|\leq\mathbb{E}(|X_{\textbf{a},i}|^{r+s})\leq\mathbb{E}(|X_{\textbf{a},i}|^{r+s+1})^{\frac{r+s}{r+s+1}}=\mathbb{E}(X_{\textbf{a},i}^{r+s+1})^{\frac{r+s}{r+s+1}}
\end{displaymath}
by Jensen's inequality, so the degree of $P(\textbf{a},i,r,s)$ is at most $$4(r+s)+\frac{3}{2}(r+s+1)\frac{r+s}{r+s+1}=\frac{11}{2}(r+s)\text{.}$$
\end{proof}

\textbf{Remark} If $r+s$ is odd, $\frac{11}{2}(r+s)$ is not an integer, so the degree of $P(\textbf{a},i,r,s)$ is strictly less than $\frac{11}{2}(r+s)$.\\

For $r+s$ even, we substitute $FM(\textbf{a},i,r,s)=\frac{P(\textbf{a},i,r,s)}{(\sum_{j=1}^da_j-1)^{\underline{4(r+s)}}}$ in (3) and multiply by the common denominator $\big(\sum_{j=1}^da_j-1\big)^{\underline{4(r+s)}}$ to obtain the polynomial identity
\begin{multline}
\Big(\sum_{j=1}^da_j-4(r+s)\Big)P(\textbf{a}+\textbf{e}_i,i,r,s)-\sum_{j=1}^da_jP(\textbf{a},j,r,s)\\
=\sum_{j=1}^ia_j\Bigg[\Big(\sum_{k=1}^da_k\Big)^4\frac{1}{2}\Big(\sum_{k=j+1}^da_k-\sum_{k=1}^{j-1}a_k\Big)rP(\textbf{a},j,r-1,s)\\
+\Big(\sum_{k=1}^da_k\Big)^4\frac{1}{2}\Big(\sum_{k=j+1}^{i}a_k-\sum_{k=1}^{j-1}a_k-\sum_{k=i+1}^{d}a_k\Big)sP(\textbf{a},j,r,s-1)\\
+\Big(\sum_{k=1}^da_k\Big)^8\frac{1}{2}\bigg(\frac{1}{2}\Big(\sum_{k=j+1}^da_k-\sum_{k=1}^{j-1}a_k\Big)\bigg)^2r(r-1)P(\textbf{a},j,r-2,s)\\
+\Big(\sum_{k=1}^da_k\Big)^8\frac{1}{2}\bigg(\frac{1}{2}\Big(\sum_{k=j+1}^{i}a_k-\sum_{k=1}^{j-1}a_k-\sum_{k=i+1}^{d}a_k\Big)\bigg)^2s(s-1)P(\textbf{a},j,r,s-2)\\
+\Big(\sum_{k=1}^da_k\Big)^8\frac{1}{2}\Big(\sum_{k=j+1}^da_k-\sum_{k=1}^{j-1}a_k\Big)\frac{1}{2}\Big(\sum_{k=j+1}^{i}a_k-\sum_{k=1}^{j-1}a_k-\sum_{k=i+1}^{d}a_k\Big)rsP(\textbf{a},j,r-1,s-1)\Bigg]\\
+\sum_{j=i+1}^da_j\Bigg[\Big(\sum_{k=1}^da_k\Big)^4\frac{1}{2}\Big(\sum_{k=j+1}^da_k-\sum_{k=1}^{j-1}a_k\Big)rP(\textbf{a},j,r-1,s)\\
+\Big(\sum_{k=1}^da_k\Big)^4\frac{1}{2}\Big(\sum_{k=j+1}^{d}a_k+\sum_{k=1}^{i}a_k-\sum_{k=i+1}^{j-1}a_k\Big)sP(\textbf{a},j,r,s-1)\\
+\Big(\sum_{k=1}^da_k\Big)^8\frac{1}{2}\bigg(\frac{1}{2}\Big(\sum_{k=j+1}^da_k-\sum_{k=1}^{j-1}a_k\Big)\bigg)^2r(r-1)P(\textbf{a},j,r-2,s)\\
+\Big(\sum_{k=1}^da_k\Big)^8\frac{1}{2}\bigg(\frac{1}{2}\Big(\sum_{k=j+1}^{d}a_k+\sum_{k=1}^{i}a_k-\sum_{k=i+1}^{j-1}a_k\Big)\bigg)^2s(s-1)P(\textbf{a},j,r,s-2)\\
+\Big(\sum_{k=1}^da_k\Big)^8\frac{1}{2}\Big(\sum_{k=j+1}^da_k-\sum_{k=1}^{j-1}a_k\Big)\frac{1}{2}\Big(\sum_{k=j+1}^{d}a_k+\sum_{k=1}^{i}a_k-\sum_{k=i+1}^{j-1}a_k\Big)rsP(\textbf{a},j,r-1,s-1)\Bigg]\\
+\text{lower order terms.}
\end{multline}
Here the lower order terms have degree at most $\frac{11}{2}(r+s)-1$. \\
Let us now set up a recurrence that is satisfied by the mixed moments of the bivariate normal distribution $(X,Y)$, with $\mathbb{E}(X)=\mathbb{E}(Y)=0$, $$Var(X)=Var(Y)=\sum_{1\leq i<	j\leq d}\big(a_ia_j^2+a_i^2a_j\big)+2\sum_{1\leq i<j<k\leq d}a_ia_ja_k\text{ and}$$ $$Cov(X,Y)=\sum_{1\leq i<j\leq d}a_ia_j^2-\sum_{1\leq i<j\leq d}a_i^2a_j\text{.}$$
Recall that the mixed moments of a multivariate normal distribution are given by Isserlis' Theorem \cite{isserlis}:
\begin{displaymath}
\mathbb{E}(X_1\cdots X_n)=\sum_P\prod_{(i,j)\in P}\mathbb{E}(X_iX_j)\text{,}
\end{displaymath}
where the sum is over all perfect matchings of $[n]$, which holds if $X_1,\ldots,X_n$ are jointly normally distributed with zero mean. In this case, this means that
\begin{equation}
\mathbb{E}(X^rY^s)=\sum_P\mathbb{E}(X^2)^{\left|P|_{[r]}\right|}\mathbb{E}(Y^2)^{\left|P|_{[s]}\right|}\mathbb{E}(XY)^{\left|P|_{K_{r,s}}\right|}\text{,}
\end{equation}
where the sum is over all perfect matchings of $[r]\amalg[s]$, $P|_{[r]}$ is the set of edges of $P$ within $[r]$, $P|_{[s]}$ is the set of edges of $P$ within $[s]$, and $P|_{K_{r,s}}$ is the set of edges of $P$ that have a vertex in $[r]$ and $[s]$ each. Now double count perfect matchings of $[r]\amalg[s]$ with the above weighting and a single vertex coloured red: On the one hand, one can pick a perfect matching and then choose a vertex to colour red. On the other hand, one can pick a vertex to colour red, then choose a vertex to match to it, and afterwards take a perfect matching of the rest. This leads to the formula
\begin{multline*}
(r+s)\mathbb{E}(X^rY^s)=\mathbb{E}(X^2)r(r-1)\mathbb{E}(X^{r-2}Y^s)\\
+\mathbb{E}(Y^2)s(s-1)\mathbb{E}(X^rY^{s-2})\\
+\mathbb{E}(XY)2rs\mathbb{E}(X^{r-1}Y^{s-1})\text{.}
\end{multline*}
\\
Now that we have these two recurrences, we can prove the following fact.
\begin{claim}
For $r+s$ even, the leading terms of $P(\textbf{a},i,r,s)$ are the same as those of $\big(\sum_{j=1}^da_j\big)^{4(r+s)}\mathbb{E}(X^rY^s)$, where $(X,Y)$ is distributed as above, for all $i\in[d]$.
\end{claim}
\begin{proof}
We proceed by induction on $r+s$. The result is true for $r=s=0$. If $r+s>0$, consider the recurrence (4). Multiply this equation by $a_i$ and sum over all $i\in[d]$, simplifying the right-hand side using the fact that, by induction hypothesis, the leading terms of $P(\textbf{a},j,r',s')$ do not depend on $j$ for $r'+s'<r+s$. After some amount of routine calculation, the result is
\begin{multline}
\sum_{i=1}^da_i\bigg(\Big(\sum_{j=1}^da_j-4(r+s)\Big)P(\textbf{a}+\textbf{e}_i,i,r,s)-\sum_{j=1}^da_jP(\textbf{a},j,r,s)\bigg)\\
=\frac{1}{8}\Big(\sum_{k=1}^da_k\Big)^9\Big(\sum_{1\leq i<	j\leq d}\big(a_i^2a_j+a_ia_j^2\big)+2\sum_{1\leq i<j<k\leq d}a_ia_ja_k\Big)r(r-1)P(\textbf{a},l,r-2,s)\\
+\frac{1}{8}\Big(\sum_{k=1}^da_k\Big)^9\Big(\sum_{1\leq i<	j\leq d}\big(a_i^2a_j+a_ia_j^2\big)+2\sum_{1\leq i<j<k\leq d}a_ia_ja_k\Big)s(s-1)P(\textbf{a},l,r,s-2)\\
+\frac{1}{8}\Big(\sum_{k=1}^da_k\Big)^9\Big(\sum_{1\leq i<j\leq d}a_ia_j^2-\sum_{1\leq i<j\leq d}a_i^2a_j\Big)2rsP(\textbf{a},l,r-1,s-1)\\
+\text{lower order terms,}
\end{multline}
for any $l\in[d]$. Notice that all the terms involving $P(\textbf{a},i,r-1,s)$ or $P(\textbf{a},i,r,s-1)$ conveniently cancel out.\\
\\
Now, by induction hypothesis, the right-hand side of (6) has degree $\frac{11}{2}(r+s)+1$. Suppose for contradiction that the degree of $P(\textbf{a},i,r,s)$ is less than $\frac{11}{2}(r+s)$, for all $i\in[d]$. The leading terms of the left-hand side cancel out, so the degree of the left-hand side is less than $\frac{11}{2}(r+s)+1$, which is absurd. Thus there exists $i_0\in[d]$ such that the degree of $P(\textbf{a},i_0,r,s)$ is exactly $\frac{11}{2}(r+s)$.
The degree of the right-hand side of (4) is at most $\frac{11}{2}(r+s)$, thus the leading terms of the left-hand side have to cancel out for $i=i_0$. By substracting (4) for $i=i_0$ from the same equation for $i=i_1\neq i_0$, we see that the leading terms of $P(\textbf{a},i,r,s)$ do not depend on $i$, so in particular the degree of $P(\textbf{a},i,r,s)$ is $\frac{11}{2}(r+s)$ for all $i\in[d]$. Recall that, for a polynomial $f(\textbf{a})=f(a_1,\dots,a_d)$,
\begin{displaymath}
f(\textbf{a}+\textbf{e}_i)=f(\textbf{a})+\frac{\partial}{\partial a_i}f(\textbf{a})+\text{lower order terms}
\end{displaymath}
and
\begin{displaymath}
\sum_{i=1}^da_i\frac{\partial}{\partial a_i}f(\textbf{a})=\deg(f)f(\textbf{a})+\text{lower order terms.}
\end{displaymath}
From these two observations, and (6), there results
\begin{multline*}
\Big(\sum_{i=1}^da_i\Big)\bigg(\frac{11}{2}(r+s)-4(r+s)\bigg)P(\textbf{a},l,r,s)\\
\begin{aligned}
=\frac{1}{8}\Big(\sum_{k=1}^da_k\Big)^9&\Big(\sum_{1\leq i<	j\leq d}\big(a_i^2a_j+a_ia_j^2\big)+2\sum_{1\leq i<j<k\leq d}a_ia_ja_k\Big)\\
&\cdot sr(r-1)P(\textbf{a},l,r-2,s)\\
+\frac{1}{8}\Big(\sum_{k=1}^da_k\Big)^9&\Big(\sum_{1\leq i<	j\leq d}\big(a_i^2a_j+a_ia_j^2\big)+2\sum_{1\leq i<j<k\leq d}a_ia_ja_k\Big)\\
&\cdot s(s-1)P(\textbf{a},l,r,s-2)\\
+\frac{1}{8}\Big(\sum_{k=1}^da_k\Big)^9&\Big(\sum_{1\leq i<j\leq d}a_ia_j^2-\sum_{1\leq i<j\leq d}a_i^2a_j\Big)2rsP(\textbf{a},l,r-1,s-1)\\
&\span+\text{lower order terms,}\\
\end{aligned}
\end{multline*}
for all $l\in[d]$.
This gives that, for all $l\in[d]$,
\begin{multline*}
(r+s)P(\textbf{a},l,r,s)\\
\begin{aligned}
=\frac{1}{12}\Big(\sum_{k=1}^da_k\Big)^8&\Big(\sum_{1\leq i<j\leq d}\big(a_i^2a_j+a_ia_j^2\big)+2\sum_{1\leq i<j<k\leq d}a_ia_ja_k\Big)\\
&\cdot r(r-1)P(\textbf{a},l,r-2,s)\\
+\frac{1}{12}\Big(\sum_{k=1}^da_k\Big)^8&\Big(\sum_{1\leq i<j\leq d}\big(a_i^2a_j+a_ia_j^2\big)+2\sum_{1\leq i<j<k\leq d}a_ia_ja_k\Big)\\
&\cdot s(s-1)P(\textbf{a},l,r,s-2)\\
+\frac{1}{12}\Big(\sum_{k=1}^da_k\Big)^8&\Big(\sum_{1\leq i<j\leq d}a_ia_j^2-\sum_{1\leq i<j\leq d}a_i^2a_j\Big)2rsP(\textbf{a},l,r-1,s-1)\\
&\span+\text{lower order terms.}\\
\end{aligned}
\end{multline*}
By induction hypothesis, this is further equal to
\begin{multline*}
\begin{aligned}
\frac{1}{12}\Big(\sum_{k=1}^da_k\Big)^8&\Big(\sum_{1\leq i<j\leq d}\big(a_i^2a_j+a_ia_j^2\big)+2\sum_{1\leq i<j<k\leq d}a_ia_ja_k\Big)\\
&\cdot r(r-1)\Big(\sum_{k=1}^da_k\Big)^{4(r-2+s)}
\mathbb{E}(X^{r-2}Y^s)\\
+\frac{1}{12}\Big(\sum_{k=1}^da_k\Big)^8&\Big(\sum_{1\leq i<j\leq d}\big(a_i^2a_j+a_ia_j^2\big)+2\sum_{1\leq i<j<k\leq d}a_ia_ja_k\Big)\\
&\cdot s(s-1)\Big(\sum_{k=1}^da_k\Big)^{4(r+s-2)}
\mathbb{E}(X^rY^{s-2})\\
+\frac{1}{12}\Big(\sum_{k=1}^da_k\Big)^8&\Big(\sum_{1\leq i<j\leq d}a_ia_j^2-\sum_{1\leq i<j\leq d}a_i^2a_j\Big)\\
&\cdot 2rs\Big(\sum_{k=1}^da_k\Big)^{4(r-1+s-1)}\mathbb{E}(X^{r-1}Y^{s-1})\\
&\span+\text{lower order terms.}\\
&=(r+s)\Big(\sum_{k=1}^da_k\Big)^{4(r+s)}\mathbb{E}(X^rY^s)+\text{lower order terms.}
\end{aligned}
\end{multline*}
\end{proof}
To conclude the proof of the main theorem, if $\textbf{m}=(m_1,\dots,m_d)$ is a vector of nonnegative integers, and $A=\{1^{m_1a},2^{m_2a},\dots,d^{m_da}\}$, then, on $S_{A,i}$,
\begin{multline*}
\mathbb{E}\left(\left(\frac{inv-\mu}{\sigma}\right)^r\left(\frac{maj-\mu}{\sigma}\right)^s\right)=\frac{\mathbb{E}((inv-\mu)^r(maj-\mu)^s)}{\sigma^{r+s}}\\
=\frac{FM(a\textbf{m},i,r,s)}{\sigma^{r+s}}+o(1)
=\frac{\mathbb{E}(X^rY^s)}{\sigma^{r+s}}+o(1)=\mathbb{E}\left(\left(\frac{X}{\sigma}\right)^r\left(\frac{Y}{\sigma}\right)^s\right)+o(1)\text{ as }a\rightarrow\infty\text{,}
\end{multline*}
for $r+s$ even and $(X,Y)$ distributed as above.. For $r+s$ odd, the degree of $P(a\textbf{m},i,r,s)$ is less than $\frac{11}{2}(r+s)$, so 
\begin{multline*}
\mathbb{E}\left(\left(\frac{inv-\mu}{\sigma}\right)^r\left(\frac{maj-\mu}{\sigma}\right)^s\right)=\frac{\mathbb{E}((inv-\mu)^r(maj-\mu)^s)}{\sigma^{r+s}}\\
=\frac{FM(a\textbf{m},i,r,s)}{\sigma^{r+s}}+o(1)=0+o(1)=\mathbb{E}\left(\left(\frac{X}{\sigma}\right)^r\left(\frac{Y}{\sigma}\right)^s\right)+o(1)\text{ as }a\rightarrow\infty\text{.}
\end{multline*}
Now $(\frac{X}{\sigma},\frac{Y}{\sigma})$ does not depend on $a$, so $\mathbb{E}\big(\big(\frac{inv-\mu}{\sigma}\big)^r\big(\frac{maj-\mu}{\sigma}\big)^s\big)\rightarrow\mathbb{E}\big(\big(\frac{X}{\sigma}\big)^r\big(\frac{Y}{\sigma}\big)^s\big)$ for all $r$, $s$.
Thus $\big(\frac{inv-\mu}{\sigma},\frac{maj-\mu}{\sigma}\big)$ on $S_{A,i}$ tends to $\big(\frac{X}{\sigma},\frac{Y}{\sigma}\big)$ in distribution as $a\rightarrow\infty$. A fortiori, the same is true on $S_A$. This is a bivariate normal distribution, with correlation coefficient $$\mathbb{E}\left(\frac{X}{\sigma}\frac{Y}{\sigma}\right)=\frac{\mathbb{E}(XY)}{\sigma^2}=\frac{\sum_{1\leq i<j\leq d}m_im_j^2-\sum_{1\leq i<j\leq d}m_i^2m_j}{\sum_{1\leq i<j\leq d}(m_im_j^2+m_i^2m_j)+2\sum_{1\leq i<j<k\leq d}m_im_jm_k}\text{.}$$ This completes the proof.
\section{Discussion}
\begin{corollary}
If $Cov(X,Y)=0$, that is $$\sum_{1\leq i<j\leq d}m_im_j^2-\sum_{1\leq i<j\leq d}m_i^2m_j=0\text{,}$$ then $inv$ and $maj$ are asymptotically independent. This happens in particular if all the $m_i$ are equal.
\end{corollary}
Also consider the special case $d=2$. In this case, the leading terms of $P(\textbf{a},i,r,s)$ have a particularly simple form.
To see this, recall from (5) that
\begin{multline*}
\begin{aligned}
\mathbb{E}(X^rY^s)&=\sum_P\mathbb{E}(X^2)^{\left|P|_{[r]}\right|}\mathbb{E}(Y^2)^{\left|P|_{[s]}\right|}\mathbb{E}(XY)^{\left|P|_{K_{r,s}}\right|}\\
&=\sum_{j=0}^r\sum_{\left\{P\big|\left|P|_{K_{r,s}}\right|=j\right\}}\mathbb{E}(X^2)^{\frac{r+s}{2}-j}\mathbb{E}(XY)^j
\end{aligned}
\end{multline*}
Now if $d=2$, $Var(X)=Var(Y)=\frac{ab(a+b)}{12}$ and $Cov(X,Y)=\frac{ab(b-a)}{12}$, so for $r+s$ even, this is equal to
\begin{multline*}
\sum_{\{j\in[r]\mid r-j even\}}\frac{s^{\underline{j}}r^{\underline{j}}}{j!}(r-j-1)!!(s-j-1)!!\left(\frac{ab(a+b)}{12}\right)^{\frac{r+s}{2}-j}\left(\frac{ab(b-a)}{12}\right)^j\\
\begin{aligned}
=r^{\underline{s}}(r-s-1)!&!\left(\frac{ab}{12}\right)^{\frac{r+s}{2}}(b-a)^s(a+b)^{\frac{r-s}{2}}\\
&\cdot{}_2F_1\left(-s/2,-s/2+1/2;(r-s)/2+1;\left(\frac{a+b}{b-a}\right)^2\right)\\
=r^{\underline{s}}(r-s-1)!&!\left(\frac{ab}{12}\right)^{\frac{r+s}{2}}(-2a)^s(a+b)^{\frac{r-s}{2}}\\
&\cdot{}_2F_1\left(-s,(r-s)/2+1/2;r-s+1;\frac{a+b}{a}\right)\\
=\left(\frac{ab}{12}\right)^\frac{r+s}{2}(a+&b)^\frac{r-s}{2}b^s(r+s-1)!!\\
&\cdot{}_2F_1(-s,(r-s)/2+1/2;-(r+s)/2+1/2;-a/b)\\
=\left(\frac{ab}{12}\right)^\frac{r+s}{2}(a+&b)^\frac{r-s}{2}\\
&\cdot\sum_{j=0}^s\frac{\binom{s}{j}}{(r-s-1)!!}(r+s-2j-1)!!(r-s+2j-1)!!(-a)^jb^{s-j}\text{.}\\
\end{aligned}
\end{multline*}
Here we used the hypergeometric identities (\cite{rahman}, (5.10)) (\cite{slater}, (1.8.10))
\begin{displaymath}
{}_2F_1(a,a+1/2;b+1/2;z^2)=(1-z)^{-2a}{}_2F_1\left(2a,b;2b;\frac{2z}{z-1}\right)\text{, and}
\end{displaymath}
\begin{displaymath}
{}_2F_1(a,-n;c;z)=(1-z)^n\frac{(a)_n}{(c)_n}{}_2F_1(c-a,-n;1-a-n;1-z)\text{,}
\end{displaymath}
and assumed, without loss of generality, that $r\geq s$. In particular, the degree of $\mathbb{E}(X^rY^s)$ is $\frac{3}{2}(r+s)$, a fact already implicitly used in the proof of \textbf{Claim 1}.

\end{document}